\documentclass[12pt,a4paper]{article}
\usepackage[utf8x]{inputenc}
\usepackage{ucs}
\usepackage{amsmath}
\usepackage{amsfonts}
\usepackage{amssymb}
\usepackage{amsthm}
\usepackage{mathrsfs}
\usepackage{fancyhdr}
\usepackage{graphicx}
\usepackage{xcolor}
\usepackage{arabtex}
\usepackage{framed}
\usepackage[top=2cm,bottom=2cm,margin=2cm]{geometry}
\usepackage{cancel}

\usepackage[dvipdfm,pdfstartview=FitH,colorlinks=true,linkcolor=blue,urlcolor=cyan,%
filecolor=green,citecolor=red]{hyperref}

\title{On the possible quantities of Fibonacci numbers that occur in some type of intervals}
\author{\sc Bakir FARHI \\
Laboratoire de Mathématiques appliquées \\
Faculté des Sciences Exactes \\
Université de Bejaia, 06000 Bejaia, Algeria \\[1mm]
\href{mailto:bakir.farhi@gmail.com}{bakir.farhi@gmail.com} \\[1mm]
\url{http://www.bakir-farhi.net}
}
\date{}

\def\N{{\mathbb N}}
\def\Z{{\mathbb Z}}
\def\F{\mathscr{F}}
\newcommand{\card}[1]{\mathrm{Card}\,#1}

\theoremstyle{plain}
\numberwithin{equation}{section}
\newtheorem{thm}{Theorem}[section]

\newtheorem{lemma}[thm]{Lemma}

\begin{document}
\maketitle
\begin{abstract}
In this paper, we show that for any integer $a \geq 2$, each of the intervals $[a^k , a^{k + 1})$ ($k \in \N$) contains either $\left\lfloor \frac{\log a}{\log\Phi}\right\rfloor$ or $\left\lceil \frac{\log a}{\log\Phi}\right\rceil$ Fibonacci numbers. In addition, the density (in $\N$) of the set of the all natural numbers $k$ for which the interval $[a^k , a^{k + 1})$ contains exactly $\left\lfloor \frac{\log a}{\log\Phi}\right\rfloor$ Fibonacci numbers is equal to $\left(1 - \left\langle \frac{\log a}{\log\Phi}\right\rangle\right)$ and the density of the set of the all natural numbers $k$ for which the interval $[a^k , a^{k + 1})$ contains exactly $\left\lceil \frac{\log a}{\log\Phi}\right\rceil$ Fibonacci numbers is equal to $\left\langle \frac{\log a}{\log\Phi}\right\rangle$. 
\end{abstract}
\noindent\textbf{MSC 2010:} 11B39, 11B05. \\
\textbf{Keywords:} Fibonacci numbers.

\section{Introduction and the main result}
Throughout this paper, if $x$ is a real number, we let $\lfloor x\rfloor$, $\lceil x\rceil$ and $\langle x\rangle$ respectively denote the greatest integer $\leq x$, the least integer $\geq x$ and the fractional part of $x$. Furthermore, we let $\card{X}$ denote the cardinal of a given finite set $X$. Finally, for any subset $A$ of $\N$, we define the \emph{density} of $A$ as the following limit (if it exists):
$$
\mathbf{d}(A) := \lim_{N \rightarrow + \infty} \frac{\card{(A \cap [1, N])}}{N} .
$$ 
It is clear that if $\mathbf{d}(A)$ exists then $\mathbf{d}(A) \in [0 , 1]$.

The Fibonacci sequence ${(F_n)}_{n \in \N}$ is defined by: $F_0 = 0$, $F_1 = 1$ and for all $n \in \N$:
\begin{equation}\label{def f_n}
F_{n + 2} = F_n + F_{n + 1}
\end{equation}
A Fibonacci number is simply a term of the Fibonacci sequence. In this paper, we denote by $\F$ the set of the all Fibonacci numbers; that is
$$
\F := \{F_n ~,~ n \in \N\} = \{0 , 1 , 2 , 3 , 5 , 8 , 13 , 21 , 34 , 55 , 89 , 144 , \dots\} .
$$
First, let us recall some important identities that will be useful in our proofs in Section \ref{sec2}. The Fibonacci sequence can be extended to the negative index $n$ by rewriting the recurrence relation \eqref{def f_n} as $F_n = F_{n + 2} - F_{n + 1}$. By induction, we easily show that for all $n \in \Z$, we have:
\begin{equation}\label{extend f_n}
F_{- n} = (-1)^{n + 1} F_n
\end{equation}
(see \cite[Chapter 5]{kosh} for the details). A closed formula of $F_n$ ($n \in \Z$) in terms of $n$ is known and it is given by:
\begin{equation}\label{binet}
F_n = \frac{1}{\sqrt{5}} \left(\Phi^n - \overline{\Phi}^n\right)
\end{equation}
where $\Phi := \frac{1 + \sqrt{5}}{2}$ is the golden ration and $\overline{\Phi} := \frac{1 - \sqrt{5}}{2} = - \frac{1}{\Phi}$. Formula \eqref{binet} is called ``the Binet Formula'' and there are many ways to prove it (see e.g., \cite[Chapter 8]{hons} or \cite[Chapter 5]{kosh}). Note that the real numbers $\Phi$ and $\overline{\Phi}$ are the roots of the quadratic equation:
$$
x^2 = x + 1 .
$$ 
More generally, we can show by induction (see e.g., \cite[Chapter 8]{hons}) that for $x \in \{\Phi , \overline{\Phi}\}$ and for all $n \in \Z$, we have:
\begin{equation}\label{x^n in function of x and f_n}
x^n = F_n x + F_{n - 1} 
\end{equation}
As remarked by Hosberger in \cite[Chapter 8]{hons}, Binet's formula \eqref{binet} immediately follows from the last formula \eqref{x^n in function of x and f_n}. On the other hand, the Fibonacci sequence satisfies the following important formula:
\begin{equation}\label{addition formula}
F_{n + m} = F_n F_{m + 1} + F_{n - 1} F_m \hspace*{2cm} (\forall n , m \in \Z)
\end{equation} 
which we call ``the addition formula''. A nice and easily proof of \eqref{addition formula} uses the formula \eqref{x^n in function of x and f_n}. We can also prove \eqref{addition formula} by using matrix calculations as in \cite[Chapter 8]{hons}. \\
As usual, we associate to the Fibonacci sequence ${(F_n)}_{n \in \Z}$ the Lucas sequence ${(L_n)}_{n \in \Z}$, defined by: $L_0 = 2$, $L_1 = 1$ and for all $n \in \Z$:
\begin{equation}\label{def l_n}
L_{n + 2} = L_n + L_{n + 1}
\end{equation}
There are many connections and likenesses between the Fibonacci sequence and the Lucas sequence. For example, we have the two following formulas (see \cite[Chapter 8]{hons} or \cite[Chapter 5]{kosh}):
\begin{eqnarray}
L_n & = & F_{n - 1} + F_{n + 1} \label{l_n in terms of f_n} \\
L_n & = & \Phi^n + \overline{\Phi}^n \label{closed formula for l_n}
\end{eqnarray}
which hold for any $n \in \Z$. For many other connections between the Fibonacci and the Lucas numbers, the reader can consult the two references cited just above.

Fibonacci's sequence plays a very important role in theoretical and applied mathematics. During the two last centuries, arithmetic, algebraic and analytic properties of the Fibonacci sequence have been investigated by several authors. One of those properties concerns the occurrence of the Fibonacci numbers in some type of intervals. For example, the French mathematician Gabriel Lam\'e (1795-1870) proved that there must be either four or five Fibonacci numbers with the same number of digits (see \cite[page 29]{posa}). A generalization of this result consists to determine the possible quantities of Fibonacci numbers that belong to an interval of the form $[a^k , a^{k + 1})$, where $a$ and $k$ are positive integers. In this direction, Honsberger \cite[Chapter 8]{hons} proved the following:\\[2mm]
\textbf{Theorem (Honsberger \cite{hons}).} \emph{%
Let $a$ and $k$ be any two positive integers. Then between the consecutive powers $a^k$ and $a^{k + 1}$ there can never occur more than $a$ Fibonacci numbers. 
}\\[2mm]
However, Honsberger's theorem gives only a upper bound for the quantity of the Fibonacci numbers in question. Furthermore, it is not optimal, because for $a = 10$, it gives a result that is weaker than Lam\'e's one. In this paper, we obtain the optimal generalization of Lam\'e's result with precisions concerning some densities. Our main result is the following:
\begin{thm}\label{thm_final}
Let $a \geq 2$ be an integer. Then, any interval of the form $[a^k , a^{k + 1})$ ($k \in \N$) contains either $\left\lfloor \frac{\log a}{\log \Phi}\right\rfloor$ or $\left\lceil \frac{\log a}{\log \Phi}\right\rceil$ Fibonacci numbers. \\
In addition, the density (in $\N$) of the set of the all natural numbers $k$ for which the interval $[a^k , a^{k + 1})$ contains exactly $\left\lfloor \frac{\log a}{\log \Phi}\right\rfloor$ Fibonacci numbers is equal to $\left(1 - \left\langle \frac{\log a}{\log \Phi}\right\rangle\right)$ and the density of the set of the all natural numbers $k$ for which the interval $[a^k , a^{k + 1})$ contains exactly $\left\lceil \frac{\log a}{\log \Phi}\right\rceil$ Fibonacci numbers is equal to $\left\langle \frac{\log a}{\log \Phi}\right\rangle$.
\end{thm}

\section{The proof of the main result}\label{sec2}
The proof of our main result needs the following lemmas:
\begin{lemma}\label{l1}
For all positive integer $n$, we have:
$$
\Phi^{n - 2} \leq F_n \leq \Phi^{n - 1} . 
$$
In addition, the left-hand side of this double inequality is strict whenever $n \geq 3$ and its right-hand side is strict whenever $n \geq 2$. 
\end{lemma}
\begin{proof}
We argue by induction on $n$. The double inequality of the lemma is clearly true for $n = 1$ and for $n = 2$. For a given integer $n \geq 3$, suppose that the double inequality of the lemma holds for any positive integer $m < n$. So it holds in particular for $m = n - 1$ and for $m = n - 2$, that is:
$$
\Phi^{n - 3} \leq F_{n - 1} \leq \Phi^{n - 2} ~~~~\text{and}~~~~ \Phi^{n - 4} \leq F_{n - 2} \leq \Phi^{n - 3} . 
$$
By adding corresponding sides of the two last double inequalities and by taking account that: $\Phi^{n - 3} + \Phi^{n - 4} = \Phi^{n - 4} (\Phi + 1) = \Phi^{n - 4} \Phi^2 = \Phi^{n - 2}$ (since $\Phi + 1 = \Phi^2$), $\Phi^{n - 2} + \Phi^{n - 3} = \Phi^{n - 1}$ (for the same reason) and $F_{n - 1} + F_{n - 2} = F_n$, we get:
$$
\Phi^{n - 2} \leq F_n \leq \Phi^{n - 1} ,
$$
which is the double inequality of the lemma for the integer $n$. This achieves this induction and confirms the validity of the double inequality of the lemma for any positive integer $n$. We can show the second part of the lemma by the same way.
\end{proof}
\begin{lemma}\label{l3}
For any integer $a \geq 2$, the real number $\frac{\log a}{\log\Phi}$ is irrational.
\end{lemma}
\begin{proof}
Let $a \geq 2$ be an integer. We argue by contradiction. Suppose that $\frac{\log a}{\log\Phi}$ is rational. Since $\frac{\log a}{\log\Phi} > 0$, we can write $\frac{\log a}{\log\Phi} = \frac{r}{s}$, where $r$ and $s$ are positive integers. This gives $\Phi^r = a^s$ and shows that $\Phi^r \in \Z$. Then, since $\overline{\Phi}^r = L_r - \Phi^r$ (according to \eqref{closed formula for l_n}), it follows that $\overline{\Phi}^r \in \Z$. But on the other hand, we have $\left\vert\overline{\Phi}^r\right\vert = {\left\vert\overline{\Phi}\right\vert}^r \in (0 , 1)$ (since $\left\vert\overline{\Phi}\right\vert \in (0 , 1)$). We thus have a contradiction which confirms that the real number $\frac{\log a}{\log \Phi}$ is irrational.
\end{proof}
\begin{lemma}\label{l4}
When the positive real number $x$ tends to infinity, then we have:
$$
\card{\left(\mathscr{F} \cap [1 , x)\right)} ~\sim~ \frac{\log{x}}{\log\Phi} .
$$
\end{lemma}
\begin{proof}
For a given $x > 1$, let $\F \cap [1 , x) = \{F_2 , F_3 , \dots , F_{h + 1}\}$ for some positive integer $h$. So we have $\card{(\F \cap [1 , x))} = h$ and: 
$$
F_{h + 1} < x \leq F_{h + 2} ,
$$
which gives:
$$
\frac{\log{F_{h + 1}}}{h \log\Phi} < \frac{\log{x}}{h \log\Phi} \leq \frac{\log{F_{h + 2}}}{h \log\Phi} .
$$
But because we have (according to the Binet formula \eqref{binet}):
$$
\lim_{h \rightarrow + \infty} \frac{\log{F_{h + 1}}}{h \log\Phi} = \lim_{h \rightarrow + \infty} \frac{\log{F_{h + 2}}}{h \log\Phi} = 1 ,
$$
it follows that $\displaystyle \lim_{h \rightarrow + \infty} \frac{\log{x}}{h \log\Phi} = 1$. Hence $h \sim_{+ \infty} \frac{\log{x}}{\log\Phi}$, as required.
\end{proof}
\begin{lemma}\label{l2}
For any $n \in \Z$, we have:
$$
F_{2 n - 1} \geq F_{n}^{2} .
$$
\end{lemma}
\begin{proof}
Let $n \in \Z$. According to the addition formula \eqref{addition formula}, we have:
$$
F_{2 n - 1} = F_{n + (n - 1)} = F_n^2 + F_{n - 1}^{2} \geq F_n^2 .
$$
The lemma is proved.
\end{proof}
\begin{lemma}\label{Formule_addition_F_L}
For all $n , m \in \Z$, we have:
$$
F_{n + m} = F_n L_m + (-1)^{m + 1} F_{n - m} .
$$
\end{lemma}
\begin{proof}
Let $n , m \in \Z$. According to the addition formula \eqref{addition formula}, we have:
\begin{equation}\label{eq_app_a}
F_{n + m} = F_n F_{m + 1} + F_{n - 1} F_m
\end{equation}
and
$$
F_{n - m} = F_n F_{- m + 1} + F_{n - 1} F_{- m} = (-1)^m F_n F_{m - 1} + (-1)^{m + 1} F_{n - 1} F_m 
$$
(according to \eqref{extend f_n}). Hence:
\begin{equation}\label{eq_app_b}
(-1)^m F_{n - m} = F_n F_{m - 1} - F_{n - 1} F_m
\end{equation}
By adding corresponding sides of \eqref{eq_app_a} and \eqref{eq_app_b}, we get:
$$
F_{n + m} + (-1)^m F_{n - m} = F_n \left(F_{m + 1} + F_{m - 1}\right) = F_n L_m 
$$
(according to \eqref{l_n in terms of f_n}). Hence:
$$
F_{n + m} = F_n L_m + (-1)^{m + 1} F_{n - m} ,
$$
as required.
\end{proof}
\begin{lemma}[the key lemma]\label{lemme clef}
For all $n , m \in \N$, satisfying $(n , m) \neq (0 , 1)$ and $n \geq m - 1$, we have:
$$
F_n \left\lfloor \Phi^m\right\rfloor \leq F_{n + m} \leq F_n \left\lceil \Phi^m\right\rceil .
$$
\end{lemma}
\begin{proof}
The double inequality of the lemma is trivial for $m = 0$. For what follows, assume that $m \geq 1$. We distinguish two cases according to the parity of $m$. \\[1mm]
\underline{1\textsuperscript{st} case:} (if $m$ is even). In this case, we have $\overline{\Phi}^m \in (0 , 1)$ (since $\overline{\Phi} \in (-1 , 0)$ and $m$ is even). Using \eqref{closed formula for l_n}, it follows that:
$$
\Phi^m = L_m - \overline{\Phi}^m \in (L_m - 1 , L_m) .
$$
Consequently, we have:
$$
\left\lfloor \Phi^m\right\rfloor = L_m - 1 ~~~~\text{and}~~~~ \left\lceil \Phi^m\right\rceil = L_m .
$$
So, for this case, we have to show that:
$$
F_n \left(L_m - 1\right) \leq F_{n + m} \leq F_n L_m .
$$
Let us show the last double inequality. According to Lemma \ref{Formule_addition_F_L}, we have:
\begin{equation}\label{eq_split}
\begin{split}
F_{n + m} &= F_n L_m + (-1)^{m + 1} F_{n - m} \\
&= F_n L_m - F_{n - m} ~~~~~~~~~~ (\text{because } m \text{ is even}) \\
&= F_n \left(L_m - 1\right) + \left(F_n - F_{n - m}\right)
\end{split}
\end{equation}
Next, since $n - m \geq -1$ (because $n \geq m - 1$ by hypothesis), then we have:
$$
F_{n - m} \geq 0 
$$
and since $n \geq n - m \geq -1$ and $(n , n - m) \neq (0 , -1)$ (because $(n , m) \neq (0 , 1)$ by hypothesis), then we have $F_n \geq F_{n - m}$; that is:
$$
F_n - F_{n - m} \geq 0 .
$$
Therefore, the second and the third equalities of \eqref{eq_split} show that:
$$
F_n \left(L_m - 1\right) \leq F_{n + m} \leq F_n L_m ,
$$
as required. \\[1mm]
\underline{2\textsuperscript{nd} case:} (if $m$ is odd). In this case, we have $\overline{\Phi}^m \in (-1 , 0)$ (because $\overline{\Phi} \in (-1 , 0)$ and $m$ is odd). Using \eqref{closed formula for l_n}, it follows that:
$$
\Phi^m = L_m - \overline{\Phi}^m \in (L_m , L_m + 1) .
$$
Consequently, we have:
$$
\left\lfloor \Phi^m\right\rfloor = L_m ~~~~\text{and}~~~~ \left\lceil \Phi^m\right\rceil = L_{m} + 1 .
$$
So, for this case, we have to show that:
$$
F_n L_m \leq F_{n + m} \leq F_n \left(L_m + 1\right) .
$$
Let us show this last double inequality. According to Lemma \ref{Formule_addition_F_L}, we have:
\begin{equation}\label{eq_split_2}
\begin{split}
F_{n + m} &= F_n L_m + (-1)^{m + 1} F_{n - m} \\
&= F_n L_m + F_{n - m} ~~~~~~~~~~ (\text{because } m \text{ is odd}) \\
&= F_n \left(L_m + 1\right) - \left(F_n - F_{n - m}\right)
\end{split}
\end{equation}
For the same reasons as in the first case, we have:
$$
F_{n - m} \geq 0 ~~~~\text{and}~~~~ F_n - F_{n - m} \geq 0 .
$$
It follows, according to the second and the third equalities of \eqref{eq_split_2}, that:
$$
F_n L_m \leq F_{n + m} \leq F_n \left(L_m + 1\right) ,
$$
as required. This completes the proof of the lemma.
\end{proof}
We are now ready to prove our main result.
\begin{proof}[Proof of Theorem \ref{thm_final}]
Let $a \geq 2$ be a fixed integer. For simplification, we put for any natural number $k$: $I_k := [a^k , a^{k + 1})$ and we put $\ell := \left\lfloor \frac{\log a}{\log\Phi}\right\rfloor$. Since the real number $\frac{\log a}{\log\Phi}$ is not an integer (according to Lemma \ref{l3}), we deduce that $\ell + 1 = \left\lceil \frac{\log a}{\log\Phi}\right\rceil$. \\
\textbullet{} First, let us show the first part of the theorem. \\
--- For $k = 0$, we have $I_k = I_0 = [1 , a)$. According to the definition of $\ell$, we have:
$$
\Phi^{\ell} \leq a < \Phi^{\ell + 1} .
$$
Hence:
\begin{equation}\label{double_inclusion}
\bigsqcup_{i = 0}^{\ell - 1} \left[\Phi^i , \Phi^{i + 1}\right) \subset I_0 \subset \bigsqcup_{i = 0}^{\ell} \left[\Phi^i , \Phi^{i + 1} \right) 
\end{equation}
(recall that the symbol $\sqcup$ denotes a disjoint union). Since, according to Lemma \ref{l1}, each interval $[\Phi^i , \Phi^{i + 1})$ ($i \in \N$) 
contains a unique Fibonacci number, it follows from \eqref{double_inclusion} that the interval $I_0$ contains at least $\ell$ Fibonacci numbers and at most $(\ell + 1)$ Fibonacci numbers, as required. \\[1mm]
--- For the following, we assume $k \geq 1$. Let $i$ denote the number of the Fibonacci numbers belonging to $I_k$ and let $F_r , F_{r + 1} , \dots , F_{r + i - 1}$ ($r \geq 2$) denote those Fibonacci numbers. We shall determine $i$. We have by definition:
\begin{equation}\label{eq_une}
F_{r - 1} < a^k \leq F_r < F_{r + 1} < \dots < F_{r + i - 1} < a^{k + 1} \leq F_{r + i}
\end{equation}
which implies that:
\begin{align}
\frac{F_{r + i - 1}}{F_r} < \frac{a^{k + 1}}{a^k} = a \label{eq_deux} \\[-5mm]
\intertext{and}
\frac{F_{r + i}}{F_{r - 1}} > \frac{a^{k + 1}}{a^k} = a \label{eq_trois}
\end{align}
On the other hand, we have:
\begin{eqnarray*}
F_{r + i - 1} & < & a^{k + 1} ~~~~~~~~~~ (\text{according to \eqref{eq_une}}) \\
& \leq & a^{2 k} ~~~~~~~~~~ (\text{since $k \geq 1$}) \\
& \leq & F_r^2 ~~~~~~~~~~ (\text{according to \eqref{eq_une}}) \\
& \leq & F_{2 r - 1} ~~~~~~~~~~ (\text{according to Lemma \ref{l2}}) .   
\end{eqnarray*}
Hence $F_{r + i - 1} < F_{2 r - 1}$. Because the sequence ${(F_n)}_{n \in \N}$ is non-decreasing, we deduce that \linebreak $r + i - 1 < 2 r - 1$, which gives $r > i$; that is $r \geq i + 1$. This then allows us to apply Lemma \ref{lemme clef} for each of the two couples $(n , m) = (r , i - 1)$ and $(n , m) = (r - 1 , i + 1)$ to obtain:
\begin{align*}
\left\lfloor \Phi^{i - 1}\right\rfloor &~\leq~ \displaystyle\frac{F_{r + i - 1}}{F_r} ~\leq~ \left\lceil \Phi^{i - 1}\right\rceil \\
\intertext{and} 
\left\lfloor \Phi^{i + 1}\right\rfloor &~\leq~ \displaystyle\frac{F_{r + i}}{F_{r - 1}} ~\leq~ \left\lceil \Phi^{i + 1}\right\rceil
\end{align*}
By comparing these last double inequalities with \eqref{eq_deux} and \eqref{eq_trois}, we deduce that:
$$
\left\lfloor \Phi^{i - 1}\right\rfloor < a < \left\lceil \Phi^{i + 1}\right\rceil .
$$
But since $a$ is an integer, it follows that:
$$
\Phi^{i - 1} < a < \Phi^{i + 1} ,
$$
which gives:
$$
\frac{\log a}{\log\Phi} - 1 < i < \frac{\log a}{\log\Phi} + 1 .
$$
Finally, since $i$ is an integer, we conclude that:
$$
i \in \left\{\left\lfloor \frac{\log a}{\log\Phi}\right\rfloor ~,~ \left\lceil \frac{\log a}{\log\Phi}\right\rceil\right\} = \left\{\ell , \ell + 1\right\},
$$
as required. \\[1mm]
\textbullet{} Now, let us show the second part of the theorem which deal with densities of subsets of the natural numbers. For a given positive integer $N$, the intervals $I_k := [a^k , a^{k + 1})$ ($0 \leq k \leq N - 1$) clearly form a partition of the interval $[1 , a^N)$. Let $A_N$ denote the number of the intervals $I_k$ ($0 \leq k \leq N - 1$) containing, each of them, exactly $\ell$ Fibonacci numbers and let $B_N$ denote the number of the intervals $I_k$ ($0 \leq k \leq N - 1$) containing, each of them, exactly $(\ell + 1)$ Fibonacci numbers. According to the first part of the theorem (shown above), we have: 
$$
A_N + B_N = N ~~~~~~\text{and}~~~~~~ \ell A_N + (\ell + 1) B_N ~=~ \card\left(\mathscr{F} \cap [1 , a^N)\right) ~\sim_{+ \infty}~ \frac{\log(a^N)}{\log{\Phi}} ~=~ \frac{\log a}{\log \Phi} N 
$$
(where the last estimate follows from Lemma \ref{l4}). So, the couple $(A_N , B_N)$ is a solution of the following linear system of two equations:
$$
\left\{
\begin{array}{lcl}
A_N + B_N & = & N \\[2mm]
\ell A_N + (\ell + 1) B_N & = & \displaystyle\frac{\log a}{\log \Phi} N + o(N)
\end{array}
\right. .
$$
By solving this system, we get:
\begin{multline*}
A_N = \left(\ell + 1 - \frac{\log a}{\log \Phi}\right) N + o(N) = \left(\left\lfloor \frac{\log a}{\log\Phi}\right\rfloor + 1 - \frac{\log a}{\log \Phi}\right) N + o(N) \\
= \left(1 - \left\langle\frac{\log a}{\log \Phi}\right\rangle\right) N + o(N)
\end{multline*}
and
$$
B_N = \left(\frac{\log a}{\log \Phi} - \ell\right) N + o(N) = \left(\frac{\log a}{\log \Phi} - \left\lfloor\frac{\log a}{\log\Phi}\right\rfloor\right) N + o(N) = \left\langle \frac{\log a}{\log\Phi}\right\rangle N + o(N) .
$$
Hence:
$$
\lim_{N \rightarrow + \infty} \frac{A_N}{N} = 1 - \left\langle\frac{\log a}{\log\Phi}\right\rangle ~~~~\text{and}~~~~ \lim_{N \rightarrow + \infty} \frac{B_N}{N} = \left\langle\frac{\log a}{\log \Phi}\right\rangle .
$$
These two limits respectively represent the density of the set of the all $k \in \N$ for which the interval $I_k$ contains exactly $\ell = \left\lfloor\frac{\log a}{\log\Phi}\right\rfloor$ Fibonacci numbers and the density of the set of the all $k \in \N$ for which $I_k$ contains exactly $(\ell + 1) = \left\lceil\frac{\log a}{\log\Phi}\right\rceil$ Fibonacci numbers. This confirms the second part of the theorem and completes this proof.
\end{proof}

\section{Numerical examples and remarks}
In this section, we apply our main result for some particular values of $a$ to deduce some interesting results.
\begin{itemize}
\item For $a = 10$, the first part of Theorem \ref{thm_final} shows that any interval of the form $[10^k , 10^{k + 1})$ ($k \in \N$) contains either $\left\lfloor \frac{\log{10}}{\log\Phi}\right\rfloor = 4$ or $\left\lceil \frac{\log{10}}{\log\Phi}\right\rceil = 5$ Fibonacci numbers. We thus find again Lam\'e's result cited in Section 1. The second part of Theorem \ref{thm_final} shows that the set of the all positive integers $k$ for which there are exactly 4 Fibonacci numbers with $k$ digits and the set of the all positive integers $k$ for which there are exactly 5 Fibonacci numbers with $k$ digits have respectively the densities $1 - \left\langle \frac{\log{10}}{\log\Phi}\right\rangle = 0.215\dots$ and $\left\langle \frac{\log{10}}{\log\Phi}\right\rangle = 0.784\dots$.
\item For $a = 7$, Theorem \ref{thm_final} shows that any interval of the form $[7^k , 7^{k + 1})$ ($k \in \N$) contains either $4$ or $5$ Fibonacci numbers. Besides, the density of the set of the all $k \in \N$ for which the interval $[7^k , 7^{k + 1})$ contains exactly $4$ Fibonacci numbers is  about $0.956$. So, there is more than a $95 \%$ chance that an arbitrary interval of the form $[7^k , 7^{k + 1})$ contains exactly $4$ Fibonacci numbers.    
\end{itemize}

\medskip

\noindent\textbf{Remark 1.} The second example above shows that for $a = 7$, the intervals $[a^k , a^{k + 1})$ ($k \in \N$) contain almost all (i.e., with a large percentage) the same quantity of Fibonacci numbers. Interestingly, we remark that $7 = L_4$ is a Lucas number. Actually, it is not difficult to show that the last property is satisfied for any other Lucas number $a \geq 7$. Indeed, if $a$ is a Lucas number (say $a = L_n$ for some $n \geq 4$) then $a$ is close to $\Phi^n$ (since $L_n = \Phi^n + \overline{\Phi}^n$, according to \eqref{closed formula for l_n}) and then $\frac{\log{a}}{\log\Phi}$ is close to $n$. It follows that one of the two densities occurring in Theorem \ref{thm_final} (that is $\langle\frac{\log{a}}{\log\Phi}\rangle$ and $1 - \langle\frac{\log{a}}{\log\Phi}\rangle$) is almost zero, which confirms that the intervals $[a^k , a^{k + 1})$ ($k \in \N$) contain almost all a same quantity of Fibonacci numbers. Taking for example $a = L_{11} = 199$, we find that more than $99.99 \%$ of the intervals $[199^k , 199^{k + 1})$ ($k \in \N$) contain exactly $11$ Fibonacci numbers. The few remaining of these intervals contain only $10$ Fibonacci numbers.

\medskip

\noindent\textbf{Remark 2.} Let $a \geq 2$ be a fixed integer. For an arbitrary natural number $k$, let $X$ denote the random variable that counts the number of the Fibonacci numbers belonging to the interval $[a^k , a^{k + 1})$. Theorem \ref{thm_final} shows that $X$ takes only two possible values $\lfloor\frac{\log{a}}{\log\Phi}\rfloor$ and $\lceil\frac{\log{a}}{\log\Phi}\rceil$ with the probabilities $(1 - \langle \frac{\log a}{\log\Phi}\rangle)$ and $\langle\frac{\log a}{\log\Phi}\rangle$, respectively. Using this, the calculations give:
\begin{eqnarray*}
\mathbb{E}(X) & = & \frac{\log{a}}{\log\Phi} \\[1mm]
\sigma(X) & = & \sqrt{\left\langle\frac{\log{a}}{\log\Phi}\right\rangle \left(1 - \left\langle\frac{\log{a}}{\log\Phi}\right\rangle\right)} ,
\end{eqnarray*} 
where $\mathbb{E}(X)$ and $\sigma(X)$ respectively denote the mathematical expectation and the standard deviation of $X$.

\end{document}